\documentclass[11pt, a4paper]{article}

\usepackage{amsmath, amssymb, amsthm, xspace, a4wide, amscd}
\usepackage[english]{babel}
\usepackage[T1]{fontenc}
\usepackage[latin1]{inputenc}

\newtheorem{thm}{Theorem}[section]
\newtheorem{lem}{Lemma}[section]
\newtheorem{nas}{Corollary}[section]

\theoremstyle{definition}
\newtheorem{ozn}{Definition}[section]
\newtheorem{exa}{Example}[section]

\newtheorem*{xrem}{Remark}

\DeclareMathOperator{\pr}{\mathsf P}

\theoremstyle{definition}
\newtheorem{ass}{Assumption}

\newcommand{\ff}{\ensuremath{\mathcal{F}}}
\newcommand{\bb}{\ensuremath{\mathcal{B}}}
\newcommand{\ww}{\ensuremath{\mathcal{W}}}
\newcommand{\jj}{\ensuremath{\mathcal{J}}}
\newcommand{\rr}{\ensuremath{\mathbb{R}}}
\newcommand{\cc}{\ensuremath{\mathbb{C}}}

\newcommand{\aaa}{{\sf A}}
\newcommand{\bbb}{{\sf B}}
\newcommand{\ccc}{{\sf C}}
\newcommand{\mm}{{\sf m}}
\newcommand{\sss}{{\sf S}}
\newcommand{\ssss}{\ensuremath{\mathcal{S}}}
\newcommand{\diam}{{\rm diam}\,}

\newcommand{\rp}{\stackrel{\pr}{\to}}

\begin{document}

\begin{center}\Large
Riemann integral of a random function and the parabolic equation with a general stochastic measure
\end{center}

\begin{center}
Vadym Radchenko \footnote{This research was supported by Alexander von Humboldt Foundation, grant no. UKR/1074615. The
author wishes to thank Prof. M.~Z\"{a}hle for fruitful discussions, and the hospitality of Jena University is
gratefully acknowledged.}
\end{center}



\begin{abstract}
For stochastic parabolic equation driven by a general stochastic measure, the weak solution is obtained. The integral
of a random function in the equation is considered as a limit in probability of Riemann integral sums. Basic properties
of such integrals are studied in the paper.
\end{abstract}

\section{Introduction}
\label{scintro} \setcounter{equation}{0}

In this paper we consider the stochastic parabolic equation, which can formally be written as
\begin{equation}\label{eqeopa}
dX(x,t)=AX(x,t)\,dt+f(x,t)\,d\mu(t),\quad X(x,0)=\xi(x),
\end{equation}
where $(x,t)\in\rr^d\times [0, T]$, $A$ is a second-order strongly elliptic differential operator, and $\mu$ is a
general stochastic measure defined on the Borel $\sigma$-algebra of $[0,T]$. For $\mu$ we assume $\sigma-$additivity in
probability only, assumptions for $A, f$ and $\xi$ are given in Section~\ref{scslpe}. Equation~\eqref{eqeopa} is
interpreted in the weak sense (see~\eqref{eqhbmuinw} below). We prove existence and uniqueness of solution.

Weak form of~\eqref{eqeopa} includes the integral of random function with respect to deterministic measure (Jordan
content). We interpret this integral as a limit in probability of Riemann integral sums. This definition of the
integral allows to interchange the order of integration with respect to deterministic and stochastic measures
(Theorem~\ref{thcisu}), that is important for solving the equation. A large part of the paper is devoted to the study
of this Riemann-type stochastic integral.

Parabolic stochastic partial differential equations (SPDEs) driven by the martingale measures had been introduced and
discussed initially in~\cite{walsh}. This approach was developed in~\cite{albeve, dalang}. Parabolic SPDEs as equations
in infinite dimensional space were studied in~\cite{dapratoz,peszab}. In these and many other papers the stochastic
noise has some distributional, integrability or martingale properties. In our paper, we consider very general class of
possible $\mu$ on $[0, T]$. On the other hand, the stochastic term in~\eqref{eqeopa} is independent of $u$. A reason is
that appropriate definition of integral of random function with respect to $\mu$ does not exist.

Some motivating examples for studying SPDEs may be found in \cite[Introduction]{dapratoz}, \cite[section 13.2]{kotele}.
For $A=\Delta$, equation~\eqref{eqeopa} describes the evolution in time of the density $X$ of some quantity such a heat
or chemical concentration in a system with random sources. In our model, the random influence can be rather general.

\section{Preliminaries}
\label{scprel} \setcounter{equation}{0}

Let $L_0=L_0(\Omega,\ff,\pr )$ be a set of all real-valued random variables defined on a complete probability space
$(\Omega, \ff, \pr )$ (equivalence classes of). Convergence in $L_0$ means the convergence in probability and is the
convergence in the quasi-norm
\[
\|\eta\|=\inf\{\delta:\pr\{|\eta|>\delta\}\le \delta\}.
\]
Note that $\|\eta_1+\eta_2\|\le \|\eta_1\|+\|\eta_2\|$. The following inequality will be used in the sequel
\begin{equation}\label{eqvaht}
\Bigl\|\sum_{k=1}^{l}c_{k}\xi_k\Bigr\|\le 8\max_{a_{k}=\pm 1}\Bigl\|\sum_{k=1}^{l}a_{k}\xi_k\Bigr\|\le
16\max_{V}\Bigl\|\sum_{k\in V}\xi_k\Bigr\|,\ |c_k|\le 1,\ \xi_k\in L_0,
\end{equation}
where the latter maximum is taken over all possible $V\subset \{1,\dots,l\}$ (see~\cite[Theorem 3]{rylwoy}).

Let $\sss$ be an arbitrary set and $\bb$ be a $\sigma$-algebra of subsets of $\sss$.

\begin{ozn}
Any $\sigma$-additive mapping $\mu:\bb\to L_0$ is called a \emph{stochastic measure}.
\end{ozn}

In other words, $\mu$ is a vector measure with values in $L_0$. We do not assume positivity or integrability for
stochastic measures. In~\cite{kwawoy} such a $\mu$ is called a general stochastic measure. In the following, $\mu$
always denotes a stochastic measure.

Examples of stochastic measures are the following. Let ${\sss}=[0, T]\subset\rr_+$, $\bb$ be the $\sigma$-algebra of
Borel subsets of $[0, T]$, and $Y(t)$ be a square integrable martingale. Then $\mu(\aaa)=\int_0^T {\bf
1}_{\aaa}(t)\,dY(t)$ is a stochastic measure. If $W^H(t)$ is a fractional Brownian motion with Hurst index~$H>1/2$ and
$f: [0, T]\to\rr$ is a bounded measurable function then $\mu(\aaa)=\int_0^T f(t){\bf 1}_{\aaa}(t)\,dW^H(t)$ is also a
stochastic measure, as follows from~\cite[Theorem~1.1]{memiva}. Some other examples may be found in~\cite[subsection
7.2]{kwawoy}. Theorem~8.3.1~\cite{kwawoy} states the conditions under which the increments of a real-valued L\'{e}vy
process generate a stochastic measure.

For deterministic measurable functions $g:\sss\to\rr$, an integral of the form $\int_{\sss}g\,d\mu$ is studied
in~\cite{radmon} (see also~\cite[Chapter 7]{kwawoy}, \cite{curbera}). The construction of this integral is standard,
uses an approximation by simple functions and is based on results of~\cite{rolewicz, talagr, turpin}. In particular,
every bounded measurable $g$ is integrable with respect to any~$\mu$. An analogue of the Lebesgue dominated convergence
theorem holds for this integral (see~\cite[Proposition 7.1.1]{kwawoy} or~\cite[Corollary 1.2]{radmon}).

For equations with stochastic measures, weak solutions of some SPDEs were obtained in~\cite{radu08}. Regularity
properties of mild solution of the stochastic heat equation were considered in~\cite{rads09}.

\section{Riemann integral of a random function}
\label{scdfri} \setcounter{equation}{0}

Let $\bbb\subset\rr^d$ be a Jordan measurable set, and $\xi:\bbb\to L_0$ be a random function. We shall say that $\xi$
has an integral on~$\bbb$ if for any sequence of partitions
\[
\bbb= \cup_{1\le k\le k_n} \bbb_{kn},\ n\ge 1,\ \max_k \diam \bbb_{kn}\to 0,\ n\to\infty,\ x_{kn}\in \bbb_{kn},
\]
the limit in probability
\begin{equation}\label{eqints}
{\rm p} \lim_{n\to\infty}\sum_{1\le k\le k_n} \xi(x_{kn})\mm(\bbb_{kn})=\int_{\bbb}\xi(x)\,dx
\end{equation}
exists. Here $\mm$ denotes the Jordan content, sets $\bbb_{kn},\ 1\le k\le k_n$, are assumed to be Jordan measurable
and have no common interior points. By mixing of different sequences of partitions, we can prove that the limit is
independent of the choice of the sequence. For deterministic $\xi$, our definition is equivalent to the definition of
the standard Riemann integral in~\cite{nikol}.

\begin{lem}\label{lmbdpr}
Let $\xi$ has an integral on~$\bbb=\prod_{k=1}^{d}[a_k,b_k]\subset\rr^d$. Then the set of values $\{\xi(x),\
x\in\bbb\}$ is bounded in probability.
\end{lem}

\begin{proof}
Is analogous to the deterministic case.
\end{proof}

For some other $\bbb\subset\rr^d$, limit~\eqref{eqints} can exists for unbounded $\xi$ (for instance, in the case
$\mm(\bbb)=0$). We use the following

\begin{ozn}\label{dfintr}
Random function $\xi$ is called \emph{integrable on~$\bbb$} if $\xi$ has an integral on~$\bbb$ and set of values
$\{\xi(x),\ x\in\bbb\}$ is bounded in probability.
\end{ozn}

Let $\tilde{\bbb}\subset\rr^d$ be an unbounded set for which there exists a sequence of Jordan measurable sets
$\bbb^{(j)}$ such that
\begin{equation}\label{eqbbbj}
\bbb^{(j)}\uparrow\tilde{\bbb},\quad \forall c>0\ \exists j:\ \tilde{\bbb}\cap\{|x|\le c\}\subset \bbb^{(j)}
\end{equation} (we call $\bbb^{(j)}$ the
exhaustive sets). We shall say that $\xi$ is integrable (in improper sense) on~$\tilde{\bbb}$, if $\xi$ is integrable
on each~$\bbb^{(j)}$, and there exists the limit in probability
\[
{\rm p} \lim_{j\to\infty} \int_{\bbb^{(j)}}\xi(x)\,dx=\int_{\tilde{\bbb}}\xi(x)\,dx  ,
\]
that is independent of choice of~$\bbb^{(j)}$.

All bounded subsets of~$\rr^d$ used in the paper are assumed to be Jordan measurable, and all unbounded sets are
assumed to be approximable by Jordan measurable sets in the sense of~\eqref{eqbbbj}. Sets in partitions are assumed to
be non-overlapping.

Obviously, if $\xi$ has the Riemann integrable paths then $\xi$ is integrable in our sense. Theorem~\ref{thcisu} below
gives other examples of integrable random functions.

Further, we establish basic properties of the integral.

\begin{lem}\label{lmseta}
Let $\xi$ be integrable on $\bbb$. Then $\xi$ is integrable on each $\aaa\subset\bbb$, and for any $\varepsilon>0$
there exists $\delta>0$ such that for all $\aaa\subset\bbb$, $\aaa= \cup_{1\le k\le k_0} \aaa_{k}$, $x_{k}\in
\aaa_{k}$, $\diam \aaa_{k}<\delta$, holds
\[
\Bigl\|\sum_{1\le k\le k_0} \xi(x_{k})\mm(\aaa_{k})-\int_{\aaa}\xi(x)\,dx\Bigr\|<\varepsilon .
\]
\end{lem}

\begin{proof}
Suppose the lemma were false. Then
\begin{multline*}
\exists\varepsilon_0>0\ \forall\delta>0\ \exists\aaa= \cup_{1\le k\le k_0} \aaa_{k}=\cup_{1\le i\le i_0} \aaa_{i}',\
\diam \aaa_{k},\ \diam \aaa_{i}'<\delta:\\
\Bigl\|\sum_{1\le k\le k_0} \xi(x_{k})\mm(\aaa_{k})-\sum_{1\le i\le i_0} \xi(x_{i}')\mm(\aaa_{i}')\Bigr\|\ge
\varepsilon_0 .
\end{multline*}
Take an arbitrary partition
\[
\bbb\setminus\aaa=\cup_{1\le j\le j_0} \ccc_{j},\quad \diam \ccc_{j}<\delta,
\]
and add
\[
\sum_{1\le j\le j_0} \xi(x_{j}'')\mm(\ccc_{j}),\quad x_{j}''\in\ccc_{j},
\]
to each of the considered sums on~$\aaa$. Thus we can get two integral sums on~$\bbb$ with arbitrary small diameters
such that the quasi-norm of their difference is greater than or equal to $\varepsilon_0$. This contradicts the
integrability of~$\xi$ on~$\bbb$.
\end{proof}

\begin{lem}
Let $\xi$ be integrable on~$\tilde{\bbb}$ in the improper sense, $\tilde{\aaa}\subset\tilde{\bbb}$. Then $\xi$ is
integrable on $\tilde{\aaa}$ (if $\tilde{\aaa}$ is an unbounded set, the integral is meant in the improper sense).
\end{lem}

\begin{proof}
Take exhaustive sets $\bbb^{(j)}\uparrow \tilde{\bbb}$, $\aaa^{(i)}\uparrow \tilde{\aaa}$. Then sets
\[
(\bbb^{(j)}\setminus\tilde{\aaa})\cup \aaa^{(i)}\uparrow \tilde{\bbb},\quad i,j\to\infty
\]
are exhaustive too, and
\begin{equation}\label{eqijab}
\int_{\tilde{\bbb}}\xi(x)\,dx={\rm p} \lim_{i,j\to\infty}\Bigl(
\int_{\bbb^{(j)}\setminus\tilde{\aaa}}\xi(x)\,dx+\int_{\aaa^{(i)}}\xi(x)\,dx\Bigr).
\end{equation}
If ${\rm p} \lim_{i\to\infty}\int_{\aaa^{(i)}}\xi(x)\,dx$ does not exist then we can choose $i,\ j\to\infty$ such that
the limit in~\eqref{eqijab} does not exist.
\end{proof}

\begin{lem}\label{lmbipr}
Let $\xi$ be integrable on~$\bbb$. Then the set of values $\Bigl\{\int_{\aaa}\xi(s)\,ds,\ \aaa\subset\bbb\Bigr\}$ is
bounded in probability.
\end{lem}

\begin{proof}
Suppose the lemma were false. Then
\[
\exists\varepsilon_0>0, \ \aaa_n\subset\bbb,\ n\ge 1:\  \Bigl\|\dfrac{1}{n}\int_{\aaa_n}\xi(s)\,ds\Bigr\|\ge
\varepsilon_0.
\]
By Lemma~\ref{lmseta}, we can choose a partition $\bbb=\cup_{1\le k\le k_0} \bbb_k$ fine enough, such that all integral
sums for partitions $\aaa_n=\cup_{1\le k\le k_0} (\aaa_n\cap \bbb_k)$ will be close enough to the integrals
on~$\aaa_n$. Thus, for all $n$, $x_{kn}\in\aaa_n\cap\bbb_k$, we get
\[
\biggl\|\dfrac{1}{n}\sum_{1\le k\le k_0}\xi(x_{kn})\mm(\aaa_n\cap\bbb_k)\biggr\|\ge \dfrac{\varepsilon_0}{2} .
\]
Since the number of summands is fixed for all $n$, we arrive at a contradiction with boundedness of $\xi$.
\end{proof}

\begin{lem}\label{lmifxi}
Let $\xi$ be integrable on~$\bbb$, $f:\bbb\to\rr$ be a deterministic uniformly continuous on~$\bbb$ function. Then
$f\xi$ is integrable on~$\bbb$.
\end{lem}

\begin{proof} Consider the difference of two integral sums of~$f\xi$
\begin{multline*}
\Bigl\|\sum_{1\le k\le k_m} f(x_{km})\xi(x_{km})\mm(\bbb_{km})-\sum_{1\le i\le i_n}
f(x_{in})\xi(x_{in})\mm(\bbb_{in})\Bigr\|\\
= \Bigl\|\sum_{1\le k\le k_m,\ 1\le i\le i_n} [f(x_{km}\xi(x_{km})-f(x_{in})\xi(x_{in})]
\mm(\bbb_{km}\cap\bbb_{in})\Bigr\|\\
\le \Bigl\|\sum_{1\le k\le k_m,\ 1\le i\le i_n} [\xi(x_{km})-\xi(x_{in})]f(x_{in})
\mm(\bbb_{km}\cap\bbb_{in})\Bigr\|\\
+ \Bigl\|\sum_{1\le k\le k_m,\ 1\le i\le i_n} [f(x_{km})-f(x_{in})]\xi(x_{km})
\mm(\bbb_{km}\cap\bbb_{in})\Bigr\|=S_1+S_2 .
\end{multline*}
From~\eqref{eqvaht} for $|f(x)|\le C$ we get
\begin{equation}\label{eqsvxb}
S_1\le 16\max_{V} \Bigl\|C\sum_{(k,i)\in V} [\xi(x_{km})-\xi(x_{in})] \mm(\bbb_{km}\cap\bbb_{in})\Bigr\|  ,
\end{equation}
where the maximum is taken over all possible sets of pairs $(k,i)$.

For example, consider
\begin{multline*}
\sum_{(k,i)\in V} \xi(x_{km})\mm(\bbb_{km}\cap\bbb_{in}) = \sum_{1\le k\le k_m} \xi(x_{km})\Bigl[\sum_{i: (k,i)\in V}
\mm(\bbb_{km}\cap\bbb_{in}) +\mm(\bbb_{km}\cap\bbb_{i'n}){\bf 1}_{x_{km}\notin (\cup_{ i:
(k,i)\in V} \bbb_{in} )}\Bigr]\\
-\sum_{1\le k\le k_m} \xi(x_{km})\mm(\bbb_{km}\cap\bbb_{i'n}){\bf 1}_{x_{km}\notin (\cup_{ i: (k,i)\in V} \bbb_{in}
)}=I_1-I_2 .
\end{multline*}
Here $\bbb_{i'n}$ is one of the sets $\bbb_{in}$, $1\le i\le i_n$, that contains $x_{km}$. (If $x_{km}$ lies on the
border of $\bbb_{i'n}$, we take it only once.) $I_1$ and $I_2$ are integral sums and, by Lemma~\ref{lmseta}, they
approximate the integrals of $\xi$ on respective sets. Therefore, for diameter small enough, $I_1- I_2$ will be close
to the integral on $\cup_{(k,i)\in V}(\bbb_{km}\cap\bbb_{in})$. Similarly, $\sum_{(k,i)\in V}
\xi(x_{in})\mm(\bbb_{km}\cap\bbb_{in})$ approximate the integral on the same set, and we make the right hand side
of~\eqref{eqsvxb} arbitrary small by choosing the diameter.

Further, for any $\alpha>0$, for diameter small enough and $\bbb_{km}\cap\bbb_{in}=\emptyset$, we have
$|f(x_{km})-f(x_{in})|<\alpha$ in~$S_2$. Inequality~\eqref{eqvaht} implies
\[
S_2\le 16 \max_{V} \Bigl\|\alpha\sum_{(k,i)\in V} \xi(x_{km}) \mm(\bbb_{km}\cap\bbb_{in})\Bigr\|.
\]
As before, we can make the sum arbitrary close to the integral on $\cup_{(k,i)\in V}(\bbb_{km}\cap\bbb_{in})$. From
Lemma~\ref{lmbipr} it follows that $S_2\to 0$ as $\alpha\to 0$.
\end{proof}

\begin{lem}\label{lminfx}
Let $\xi$ be integrable on~$\bbb$, $f:\bbb\to\rr$ be a deterministic uniformly continuous on~$\bbb$ function,
$|f(x)|\le C$. Then
\[
\Bigl\|\int_{\bbb}f(x)\xi(x)\,dx\Bigr\|\le 16\sup_{\aaa\subset\bbb}\Bigl\|C\int_{\aaa}\xi(x)\,dx\Bigr\| .
\]
\end{lem}

\begin{proof}
The inequality for respective integral sums follows from~\eqref{eqvaht}. Further, we pass to the limit and apply
Lemmas~\ref{lmseta} and~\ref{lmifxi}.
\end{proof}

\begin{lem}\label{lmfnxi}
Let $\xi$ be integrable on~$\bbb$, $f_n:\bbb\to\rr,\ n\ge 1$ be a deterministic uniformly continuous on~$\bbb$
functions, $\sup_{x\in \bbb}|f_n(x)|\to 0,\ n\to\infty$. Then
\[
\int_{\bbb} f_n(x)\xi(x)\,dx\rp 0,\quad n\to\infty.
\]
\end{lem}

\begin{proof}
The statement follows from Lemmas~\ref{lmbipr} and~\ref{lminfx}.
\end{proof}

\begin{lem}\label{lmifxn}
Let $\xi$ be integrable on an unbounded set $\tilde{\bbb}$ in improper sense, $f:\tilde{\bbb}\to\rr$ be a deterministic
bounded uniformly continuous on~$\tilde{\bbb}$ function. Then $f\xi$ is integrable on $\tilde{\bbb}$ in improper sense.
\end{lem}

\begin{proof}
For $\bbb^{(j)}\uparrow\tilde{\bbb}$ and $|f(x)|\le C$ Lemma~\ref{lminfx} implies
\begin{equation}\label{eqifxn}
\Bigl\|\int_{\bbb^{(j)}\setminus\bbb^{(i)}} f(x)\xi(x)\,dx\Bigr\|\le
16\sup_{\aaa\subset(\bbb^{(j)}\setminus\bbb^{(i)})}\Bigl\|C\int_{\aaa}\xi(x)\,dx\Bigr\| .
\end{equation}
If the left hand side of~\eqref{eqifxn} does not tend to 0 as~$i,\ j\to\infty$, then we can construct a sequence of
bounded sets~$\ccc^{(j)}\uparrow\tilde{\bbb}$ such that the sequence $\int_{\ccc_j}\xi(x)\,dx,\ j\ge 1$, is
non-fundamental.
\end{proof}

\begin{lem}\label{lmfnxr}
Let $\xi$ be integrable on unbounded set $\tilde{\bbb}$ in improper sense, $f_n:\tilde{\bbb}\to\rr$ be a deterministic
bounded uniformly continuous on~$\tilde{\bbb}$ functions, $\sup_{n\ge 1,x\in\tilde{\bbb}}|f_n(x)|=C<\infty$,
$\sup_{x\in \bbb}|f_n(x)|\to 0,\ n\to\infty$ for all bounded $\bbb\subset\tilde{\bbb}$. Then
\[
\int_{\tilde{\bbb}} f_n(x)\xi(x)\,dx\rp 0,\quad n\to\infty.
\]
\end{lem}

\begin{proof}
Suppose the lemma is false. Applying Lemma~\ref{lmfnxi}, one can find $\varepsilon_0>0$, subsequence $f_{n_j},\ j\ge
1$, and bounded disjoint sets $\bbb_j\subset(\tilde{\bbb}\cap\{|x|\ge j\})$ such that $\Bigl\|\int_{\bbb_j}
f_{n_j}(x)\xi(x)\,dx\Bigr\|> \varepsilon_0$. From Lemma~\ref{lminfx} it follows that there exist bounded disjoint sets
$\aaa_j\subset(\tilde{\bbb}\cap\{|x|\ge j\})$ such that $\Bigl\|C \int_{\aaa_j} \xi(x)\,dx\Bigr\|> (\varepsilon_0 /
16)$. This contradicts the integrability of $\xi$ on $\tilde{\bbb}$.
\end{proof}

Note that the stochastic continuity of $\xi$ does not imply the integrability.

\begin{exa} Consider $\bbb=[0, 1]$, $\xi_k(\omega)=5^k {\bf 1}_{F_k},\ k\ge 1$, were $\pr(F_k)=\dfrac{1}{k},\ F_k$
are independent. Set
\begin{multline*}
\xi(0)=0,\quad \xi(x)=\xi_k,\ 2^{-2k-1}\le x\le 2^{-2k},\\
 \xi(x)=2^{2k+2}((2^{-2k-1}-x)\xi_{k+1}+(x-2^{-2k-2})\xi_k),\
 2^{-2k-2}\le x\le 2^{-2k-1} .
\end{multline*}
Taking all possible finite unions $\aaa=\cup_k [2^{-2k-1}, 2^{-2k}]$, we see that the values $\int_{\aaa}\xi(x)\,dx$
are not bounded in probability. By Lemma~\ref{lmbipr}, $\xi$ is not integrable on~$[0, 1]$.
\end{exa}

\section{Interchange of the order of integration}
\label{scintc} \setcounter{equation}{0}

\begin{thm}\label{thcisu} Let $\mu$ be a stochastic measure on~$(\sss, \bb)$,
$\bbb\subset\rr^d$ be a bounded set. Assume that $h(x, s):\bbb\times\sss\to\rr$ is a measurable deterministic function
which is Riemann integrable on~$\bbb$ for each fixed~$s$, and $|h(x, s)|\le g(s)$, where $g:\sss\to\rr$ is integrable
on~$\sss$ with respect to~$d\mu(s)$. Then the random function $\xi(x)=\int_{\sss} h(x, s)\,d\mu(s)$ is integrable on
$\bbb$, and
\begin{equation}\label{eqcisu}
\int_{\bbb}\,dx\int_{\sss} h(x, s)\,d\mu(s)=\int_{\sss}\,d\mu(s)\int_{\bbb} h(x, s)\,dx .
\end{equation}
\end{thm}

\begin{proof} From the inequality $|h(x, s)|\le g(s)$ and~\eqref{eqvaht} it follows that values of $\xi$ are bounded in
probability (see Lemma~1.1 and Theorem~1.3~\cite{radmon}). Integral sums of $\int_{\bbb} \xi(x)\,dx$ have the form
\begin{multline*}
\sum_{1\le k\le k_n} \mm(\bbb_{kn}) \int_{\sss} h(x_{kn}, s)\,d\mu(s)=
\int_{\sss} g_n(s)\,d\mu(s),\\
g_n(s)=\sum_{1\le k\le k_n} h(x_{kn}, s)\mm(\bbb_{kn})\to \int_{\bbb}  h(x, s)\,dx .
\end{multline*}
Boundedness condition of $h$ and the analogue of the Lebesgue theorem~\cite[Proposition 7.1.1]{kwawoy} for the integral
with respect to $d\mu(s)$ imply the statement.
\end{proof}

\begin{nas}\label{thcisur}
Let $\mu$ be a stochastic measure on~$(\sss, \bb)$, $\tilde{\bbb}\subset\rr^d$ be an unbounded set. Assume that $h(x,
s):\tilde{\bbb}\times\sss\to\rr$ is a measurable deterministic function which is Riemann integrable on~$\tilde{\bbb}$
in improper sense for each fixed~$s$, and $|h(x, s)|\le g(s),\ \int_{\tilde{\bbb}}|h(x, s)|\,dx=g_1(s)$, where $g,\
g_1:\sss\to\rr$ are integrable on~$\sss$ with respect to~$d\mu(s)$. Then the random function $\xi(x)=\int_{\sss} h(x,
s)\,d\mu(s)$ is integrable on~$\tilde{\bbb}$ in improper sense, and
\begin{equation}\label{eqcisun}
\int_{\tilde{\bbb}}\,dx\int_{\sss} h(x, s)\,d\mu(s)=\int_{\sss}\,d\mu(s)\int_{\tilde{\bbb}} h(x, s)\,dx .
\end{equation}
\end{nas}

\begin{proof}
For bounded sets $\bbb^{(j)}\uparrow \tilde{\bbb}$, Theorem~\ref{thcisu} implies
\[
\int_{\bbb^{(j)}}\,dx\int_{\sss} h(x, s)\,d\mu(s)= \int_{\sss}\,d\mu(s)\int_{\bbb^{(j)}} h(x, s)\,dx .
\]
Further, we use the analogue of the Lebesgue theorem and integrability of~$g_1$.
\end{proof}

\begin{thm}\label{thmxss}
Let $\bbb\subset\rr^{d},\ \sss\subset\rr^{m}$ be a bounded sets, random function $\xi(x, s):\bbb\times\sss\to L_0$ be
integrable on~$\bbb\times\sss$ with respect to~$dx\times ds$ and be integrable on~$\sss$ with respect to~$ds$ for each
fixed~$x$. Then
\begin{equation}\label{eqmxst}
\int_{\bbb\times\sss} \xi(x, s)\,dx\times ds=\int_{\bbb}\,dx\int_{\sss} \xi(x, s)\,ds .
\end{equation}
\end{thm}

\begin{proof}
Integral sums of integral with respect to~$dx$ in~\eqref{eqmxst} has the form
\begin{equation}\label{eqmxss}
\sum_{1\le k\le k_0} \mm(\bbb_{k})\int_{\sss} \xi(x_{k}, s)\,ds .
\end{equation}
Each integral in~\eqref{eqmxss} may be approximated by sums of the form~$\sum_{1\le i\le i_0} \mm(\sss_{i})\xi(x_{k},
s_{i})$. Thus, the sums
\[
\sum_{1\le k\le k_0} \sum_{1\le i\le i_0} \mm(\bbb_{k})\mm(\sss_{i})\xi(x_{k}, s_{i}) .
\]
will approximate the right hand side of~\eqref{eqmxss}. But they are the integral sums for the integral with respect
to~$dx\times ds$ in~\eqref{eqmxst}, and will be close to the left hand side of~\eqref{eqmxst} for sufficiently small
diameters of $\bbb_{k}\times \sss_{i}$.
\end{proof}

\begin{nas}\label{crinsr}
Let $\sss\subset\rr^{m}$ be a bounded set, $\tilde{\bbb}\subset\rr^d$ be an unbounded set. Assume that the random
function $\xi(x, s):\tilde{\bbb}\times\sss\to L_0$ is integrable on~$\tilde{\bbb}\times\sss$ with respect to~$dx\times
ds$ in improper sense,  is integrable on~$\tilde{\bbb}$  with respect to~$dx$  in improper sense for each fixed~$s$,
and is integrable on~$\sss$  with respect to~$ds$ for each fixed~$x$. Then
\begin{equation}\label{eqinsr}
\int_{\tilde{\bbb}\times\sss} \xi(x, s)\,dx\times ds=\int_{\sss} ds\int_{\tilde{\bbb}} \xi(x,
s)\,dx=\int_{\tilde{\bbb}}\,dx\int_{\sss} \xi(x, s)\,ds .
\end{equation}
\end{nas}

\begin{proof}
Consider exhaustive sets $\bbb^{(j)}\uparrow\tilde{\bbb}$. For the first of the repeated integrals~\eqref{eqinsr}, the
integral sums has the form
\begin{equation}\label{eqmxsr}
\sum_{1\le i\le i_0} \mm(\sss_{i})\int_{\tilde{\bbb}} \xi(x, s_{i})\,dx
\end{equation}
The integrals in~\eqref{eqmxsr} can be approximated by $\int_{\bbb^{(j)}} \xi(x, s_{i})\,dx$, and the last integral is
the limit of sums
\[
\sum_{1\le k\le k_0} \mm(\bbb^{(j)}_{k})\xi(x_{k}^{(j)}, s_{i}) .
\]
If integral sums~\eqref{eqmxsr} does not converge, then we can construct a non-convergent sequence of sums
\[
\sum_{1\le i\le i_0} \sum_{1\le k\le k_0} \mm(\sss_{i})\mm(\bbb^{(j)}_{k})\xi( x_{k}^{(j)}, s_{i}),
\]
and this contradicts the integrability of~$\xi$ on~$\sss\times\tilde{\bbb}$.

Further, by Theorem~\ref{thmxss}, for each~$j$ we have
\[
\int_{\bbb^{(j)}\times\sss} \xi(x, s)\,dx\times ds=\int_{\bbb^{(j)}} dx\int_{\sss} \xi(x, s)\,ds .
\]
The left hand side has the limit in probability as $j\to \infty$. Hence, the right hand side has the limit, and the
second equality of~\eqref{eqinsr} holds.
\end{proof}

\section{Integration by parts}
\label{scinpr} \setcounter{equation}{0}

To solve the parabolic stochastic equation, we need the following two lemmas.

\begin{lem}\label{lmetxi}
Let a random function $\xi(u):[0, s]\to L_0$ be integrable on~$[0, s]$. Then $\eta(u)=\int_0^u \xi(v)\,dv$ is
integrable on~$[0, s]$, and
\[
\int_0^s du \int_0^u \xi(v)\,dv=\int_0^s (s-v)\xi(v)\,dv .
\]
\end{lem}

\begin{proof}
By Lemma~\ref{lmifxi}, the function $(s-v)\xi(v)$ is integrable, by Lemma~\ref{lmseta} $\eta(u)$ is well defined. The
integral sum of~$\int_0^s \eta(u)\,du$ has the form
\begin{equation}\label{eqbieta}
\sum_{1\le k\le k_0}\mm(\bbb_{k})\int_0^{u_{k}}\xi(v)\,dv,\quad u_{k}\in\bbb_{k} .
\end{equation}
We can take a new partition $[0, s]=\cup_{1\le i\le i_0} \ccc_i$ such that each integral  $\int_0^{u_{k}}\xi(v)\,dv$ be
close enough to integral sum with this partition (Lemma~\ref{lmseta}). Thus we can approximate~\eqref{eqbieta}
arbitrary closely by the sum
\begin{equation}\label{eqbietk}
\sum_{1\le k\le k_0}\mm(\bbb_{k})\sum_{1\le i\le i_0}\mm(\ccc_{i}\cap[0, {u_{k}}])\xi(v_{i}),\quad v_{i}\in\ccc_{i} .
\end{equation}
For $\int_0^s (s-v)\xi(v)\,dv$, take the integral sum
\begin{equation}\label{eqbieti}
\sum_{1\le i\le i_0}\mm(\ccc_{i})(s-v_{i})\xi(v_{i}) .
\end{equation}
The difference of \eqref{eqbieti} and \eqref{eqbietk} is equal to
\begin{equation}\label{mldfxi}
\sum_{1\le i\le i_0}\xi(v_{i})[\mm(\ccc_{i})(s-v_{i})- \mm(\ccc_{i})\sum_{k:\ \ccc_{i}<\bbb_{k}}\mm(\bbb_{k})  -
\sum_{k:\ \ccc_{i}\cap\bbb_{k}\ne\emptyset}\mm(\bbb_{k})\mm(\ccc_{i}\cap[0, {u_{k}}])] .
\end{equation}
Notation $\ccc_{i}<\bbb_{k}$ means that $v<u$ for all $v\in\ccc_{i}, u\in\bbb_{k}$. We have
\[
0\le(s-v_{i})-\sum_{k:\ \ccc_{i}<\bbb_{k}}\mm(\bbb_{k})\le \max_i\diam \ccc_{i}+\max_k\diam \bbb_{k} .
\]
The last sum of~\eqref{mldfxi} is not greater than
\[
\mm(\ccc_{i}) \sum_{k:\ \ccc_{i}\cap\bbb_{k}\ne\emptyset}\mm(\bbb_{k})\le \mm(\ccc_{i})(\max_i\diam \ccc_{i}+2
\max_k\diam \bbb_{k}) .
\]
Therefore, value~\eqref{mldfxi} may be written in the form $\sum_{1\le i\le i_0}\xi(v_{i})\mm(\ccc_{i})\alpha_{i}$,
where $\alpha_{i}\to 0$ as $\diam \ccc_{i},\ \diam \bbb_{k}\to 0$. From~\eqref{eqvaht} we obtain
\begin{equation}\label{eqxiiv}
\Bigl\|\sum_{1\le i\le i_0}\xi(v_{i})\mm(\ccc_{i})\alpha_{i}\Bigr\|\le 16\max_{V} \Bigl\|\max_{i}|\alpha_{i}|\sum_{i\in
V}\xi(v_{i})\mm(\ccc_{i})\Bigr\|.
\end{equation}
The sums $\sum_{i\in V}\xi(v_{i})\mm(\ccc_{i})$ are close to respective integrals for $\diam\ccc_i$ small enough
(Lemma~\ref{lmseta}) and values of integrals are bounded in probability (Lemma~\ref{lmbipr}). Therefore, the left hand
side of~\eqref{eqxiiv} tends to zero as $\max_{i}|\alpha_{i}|\to 0$.
\end{proof}

\begin{lem}\label{lmintp}
Let a random function $\xi(u):[0, s]\to L_0$ be integrable on~$[0, s]$, $f\in\cc^{(1)}([0, s])$ be a deterministic
function. Then
\begin{equation}\label{eqinfc}
f(s)\int_{0}^s \xi(u)\,du=\int_{0}^s  f(u)\xi(u)\,du+\int_{0}^s  f'(u)\,du\int_{0}^u \xi(v)\,dv .
\end{equation}
\end{lem}

\begin{proof}
From Lemmas~\ref{lmifxi} and~\ref{lmetxi} it follows that the random functions $\zeta_1(u)=f(u)\xi(u),\
\zeta_2(u)=f'(u)\int_{0}^u \xi(v)\,dv$ are integrable on $[0, s]$. First, let us show that for
$0=u_0<u_1<\dots<u_{k_0}=s,\ \alpha=\max_{k}|u_{k}-u_{k-1}|$, we have
\begin{equation}\label{mlsfuk}
\sum_{1\le k\le k_0} (f(u_{k})-f(u_{k-1}))\int_0^{u_{k}} \xi(v)\,dv\rp\int_0^s f'(u)\,du\int_0^u \xi(v)\,dv,\quad
\alpha\to 0.
\end{equation}
Applying the Lagrange formula and integrability of~$\zeta_2$, for some $\tilde{u}_{k}\in (u_{k-1}, u_{k})$ we obtain
\begin{multline*}
\sum_{1\le k\le k_0}(f(u_{k})-f(u_{k-1}))\int_0^{u_{k}} \xi(v)\,dv
=\sum_{1\le k\le k_0} f'(\tilde{u}_{k})(u_{k}-u_{k-1})\int_0^{u_{k}} \xi(v)\,dv\\
=\sum_{1\le k\le k_0} f'(\tilde{u}_{k})(u_{k}-u_{k-1})\int_0^{\tilde{u}_{k}}
\xi(v)\,dv+\sum_{1\le k\le k_0} f'(\tilde{u}_{k})(u_{k}-u_{k-1})\int_{\tilde{u}_{k}}^{u_{k}} \xi(v)\,dv,\\
\sum_{1\le k\le k_0} f'(\tilde{u}_{k})(u_{k}-u_{k-1})\int_0^{\tilde{u}_{k}} \xi(v)\,dv\rp \int_0^s f'(u)\,du\int_0^u
\xi(v)\,dv,\quad \alpha\to 0 .
\end{multline*}
For $C_1=\max_{u}|f'(u)|$, from~\eqref{eqvaht} we have
\[
\Bigl\|\sum_{1\le k\le k_0} f'(\tilde{u}_{k})(u_{k}-u_{k-1})\int_{\tilde{u}_{k}}^{u_{k}} \xi(v)\,dv\Bigr\| \le 16
\max_{V}\Bigl\|C_1\alpha\sum_{k\in V} \int_{\tilde{u}_{k}}^{u_{k}} \xi(v)\,dv\Bigr\|\le 16\sup_{\aaa}
\Bigl\|C_1\alpha\int_{\aaa} \xi(v)\,dv\Bigr\|.
\]
From Lemma~\ref{lmbipr} it follows that the last value tends to 0 as $\alpha\to 0$. Therefore, \eqref{mlsfuk} is
proved.

Integrability of $\zeta_1$ implies
\begin{multline*}
\sum_{1\le k\le k_0} f(u_{k-1})\int_{u_{k-1}}^{u_{k}} \xi(v)\,dv=\sum_{1\le k\le k_0} f(u_{k-1})
\xi(u_{k-1})(u_{k}-u_{k-1})\\
+\sum_{1\le k\le k_0} f(u_{k-1})\int_{u_{k-1}}^{u_{k}} (\xi(v)-\xi(u_{k-1}))\,dv ,\\
\sum_{1\le k\le k_0} f(u_{k-1})\xi(u_{k-1})(u_{k}-u_{k-1})\rp \int_{0}^s f(u)\xi(u)\,du,\quad \alpha\to 0 .
\end{multline*}
For $C_0=\max_{u}|f(u)|$, from~\eqref{eqvaht} we get
\begin{multline*}
\Bigl\|\sum_{1\le k\le k_0}  f(u_{k-1})\int_{u_{k-1}}^{u_{k}} (\xi(v)-\xi(u_{k-1}))\,dv \Bigr\|
\le 16 \max_{V}\Bigl\|C_0\sum_{k\in V} \int_{u_{k-1}}^{u_{k}} (\xi(v)-\xi(u_{k-1}))\,dv \Bigr\|\\
=16 \max_{V}\Bigl\|C_0(\int_{\cup_{k\in V} [u_{k-1}, u_{k}]} \xi(v)\, dv - \sum_{k\in V} \xi(u_{k-1})(u_{k}-u_{k-1}) )
\Bigr\| .
\end{multline*}
By Lemma~\ref{lmseta}, the last value tends to 0 as $\alpha\to 0$.

Further, we take the obvious equality
\[
f(s)\int_{0}^s \xi(v)\,dv=\sum_{1\le k\le k_0} (f(u_{k})-f(u_{k-1}))\int_0^{u_{k}} \xi(v)\,dv +\sum_{1\le k\le k_0}
f(u_{k-1})\int_{u_{k-1}}^{u_{k}} \xi(v)\,dv
\]
and pass to the limit as $\alpha\to 0$.
\end{proof}

\section{Parabolic equation with a general stochastic measure}
\label{scslpe} \setcounter{equation}{0}

Consider the differential operator
\[
Ag(x)=\sum_{1\le i,j\le d}a_{ij}(x)\dfrac{\partial^2 g(x)}{\partial x_i\partial x_j}+ \sum_{1\le i\le
d}b_{i}(x)\dfrac{\partial g(x)}{\partial x_i}+c(x)g(x),
\]
where $g:\rr^d\to\rr$ and $a_{ij}=a_{ji}$. Suppose that $A$ is strongly elliptic in~$\rr^d$ (see~(4.5)~\cite{ilkaol}).

\begin{ass}\label{ass1} All functions $a_{ij},\ b_{i},\ c,\ \dfrac{\partial a_{ij} }{\partial x_i}$, $\dfrac{\partial^2
a_{ij} }{\partial x_i \partial x_j}$, $\dfrac{\partial b_{i} }{\partial x_i}$ are bounded and H\"{o}lder continuous
in~$\rr^d$.
\end{ass}

From now on let $\mu$ be a stochastic measure on Borel subsets of~$[0, T]$.

We will  study the equation
\[
dX(x, t)=AX(x, t)\,dt+f(x, t)d\mu(t),\quad X(x, 0)=\xi(x), \tag{\ref{eqeopa}}
\]
where $X:\rr^d\times [0, T]\to L_0$ is an unknown random function.

We consider~\eqref{eqeopa} in the weak sense, i.e.
\begin{multline}
\label{eqhbmuinw} \int_{\rr^d}X(x, t)\varphi(x)\, dx=\int_{\rr^d}\xi(x)\varphi(x)\, dx\\
+\int_{\rr^d} A^*\varphi(x)\,dx\int_0^t X(x, s)\,ds+\int_{[0, t]}\,d\mu(s)\int_{\rr^d} f(x, s)\varphi(x)\, dx
\end{multline}
for all test functions $\varphi\in\ssss(\rr^d)$ (rapidly decreasing Schwartz functions from $\cc^{\infty}(\rr^d)$). For
each fixed $t\in [0, T]$ equality~\eqref{eqhbmuinw} holds a.s. Integrals of random functions with respect to $dx$ and
$ds$ are considered in Riemann sense (see section~\ref{scdfri}), and $A^*$ denotes the adjoint operator of $A$.

\begin{ass}\label{ass2} $\xi:\rr^d\to L_0$ is such that $\xi(\cdot, \omega)$ is continuous and bounded in $\rr^d$
for each fixed~$\omega\in\Omega$. \end{ass}

\begin{ass}\label{ass3} $f:\rr^d\times [0, T]\to\rr$ is Borel measurable, $\sup_{t}|x|^{-k} |f(x, t)|\to 0,\
|x|\to\infty$, for some $k>0$, $f(x, \cdot)$ is continuous and bounded in $\rr^d$ for each fixed~$t\in [0, T]$.
\end{ass}

By Theorem~1 \S 4~\cite{ilkaol}, under Assumption~\ref{ass1}, the equation $\partial g / \partial t= Ag$ has a
fundamental solution~$p(x, y, t-s)$ (recall that coefficients of~$A$ do not depend on~$t$). The following estimate is
well known:
\[
|p(x, y, t)|\le C_1 t^{-d/2}\exp\{-C_2|x-y|^2/t\}
\]
(see, for example, (4.16)~\cite{ilkaol}). Consider the semigroup
\[
S(t)g(x)=\int_{\rr^d} p(x, y, t)g(y)\,dy,\quad t>0,\quad S(0)g(x)=g(x) .
\]
Theorem~2 \S4~\cite{ilkaol} implies that for any continuous bounded~$g$
\begin{equation}\label{eqstga}
S(t)g(x)=g(x)+A\int_0^t [S(s)g(x)] \,ds.
\end{equation}

\begin{thm}\label{thhbmuinw}
Suppose Assumptions \ref{ass1}--\ref{ass3} hold. Then the random function
\begin{equation}
\label{eqhbmusl} X(x, t)=S(t)\xi(x)+\int_{[0, t]} [S(t-s)f(x,s)]\,d\mu(s)
\end{equation}
is the solution of~\eqref{eqhbmuinw}.

In addition, suppose the operator $A$ is self-adjoint, $X(x, t)$ satisfies~\eqref{eqhbmuinw}, is integrable
on~$\rr^d\times[0, T]$ with respect to~$dx\times dt$, is integrable on~$\rr^d$ with respect to~$dx$ for each fixed~$t$,
and is integrable on~$[0, T]$ with respect to~$dt$ for each fixed~$x$. Then $X(x, t)$ is given by~\eqref{eqhbmusl}.
\end{thm}

\begin{proof} From~\eqref{eqstga} it follows that for
$X_1(x, t)=S(t)\xi(x)$ and~$f=0$ equality~\eqref{eqhbmuinw} holds. For $X_2(x, t)=\int_{[0, t]}
[S(t-s)f(x,s)]\,d\mu(s)$ we have
\begin{multline*}
\int_{\rr^d} A^*\varphi(x)\,dx\int_0^t X_2(s)\,ds+\int_{[0, t]}\,d\mu(s)\int_{\rr^d} f(x,s)\varphi(x)\, dx\\
=\int_{\rr^d} A^*\varphi(x)\,dx\int_0^t \,ds\int_{[0, s]} [S(s-u)f(x,u)]\,d\mu(u)
+\int_{[0, t]}\,d\mu(s)\int_{\rr^d}f(x,s)\varphi(x)\, dx\\
\stackrel{\eqref{eqcisu}}{=}\int_{\rr^d} A^*\varphi(x)\,dx\int_{[0, t]} \,d\mu(u)\int_u^t [S(s-u)f(x,u)]\,ds
+\int_{[0, t]}\,d\mu(s)\int_{\rr^d}f(x,s)\varphi(x)\, dx\\
\stackrel{\eqref{eqcisun}}{=}\int_{[0, t]} \,d\mu(u) \int_{\rr^d} A^*\varphi(x)\,dx\int_u^t [S(s-u)f(x,u)]\,ds
+\int_{[0, t]}\,d\mu(s)\int_{\rr^d}f(x,s)\varphi(x)\, dx\\
=\int_{[0, t]} \,d\mu(u) \int_{\rr^d} \varphi(x)\,dx  A\int_u^t [S(s-u)f(x,u)]\,ds
+\int_{[0, t]}\,d\mu(s)\int_{\rr^d} f(x,s)\varphi(x)\, dx\\
\stackrel{\eqref{eqstga}}{=}\int_{[0, t]} \,d\mu(u) \int_{\rr^d} \varphi(x)\,dx ([S(t-u)f(x,u)]-f(x,u))
+\int_{[0, t]}\,d\mu(s)\int_{\rr^d} f(x,s)\varphi(x)\, dx\\
=\int_{[0, t]} \,d\mu(u) \int_{\rr^d} \varphi(x) [S(t-u)f(x,u)]\,dx
\stackrel{\eqref{eqcisun}}{=}\int_{\rr^d}\varphi(x)\, dx \int_{[0, t]} [S(t-s)f(x,s)]\,d\mu(s)\\
= \int_{\rr^d} X_2(x,t)\varphi(x)\, dx .
\end{multline*}
Therefore~\eqref{eqhbmuinw} holds for~$X=X_1+X_2$.

Finally, we will prove the uniqueness of the solution. Section~\ref{scintc} implies that random function $X$ given
by~\eqref{eqhbmusl} is integrable. Thus, it is enough to prove that the equation
\begin{equation}\label{eqzuns}
\int_{\rr^d}X(x, t)\varphi(x)\, dx=\int_{\rr^d} A^*\varphi(x)\, dx\int_0^t X(x, s)\,ds
\end{equation}
has only the zero solution provided that~$A=A^*$.

For $\varphi\in\ssss(\rr^d)$ and $0<s<t$ set $ \psi_{t,s}(x)=S(t-s)\varphi(x)$. Then
\[
\psi_{t,s}\in\ssss(\rr^d),\quad A\psi_{t,s}+\dfrac{\partial }{\partial s}\psi_{t,s}=0,\quad \psi_{t,s}\to\varphi
\]
uniformly on any bounded subset of~$\rr^d$ as $t\downarrow s$ (see~(4.13)~\cite{ilkaol}), and we get
\begin{multline*}
\int_{\rr^d}X(x, t)\psi_{t,s}(x)\, dx\stackrel{\eqref{eqzuns}}{=} \int_{\rr^d} A\psi_{t,s}(x)\, dx\int_0^s X(x, u)\,du\\
\stackrel{\eqref{eqinfc}}{=}\int_{\rr^d} \, dx\int_0^s A\psi_{t,u}(x)X(x, u)\,du
+ \int_{\rr^d} \, dx\int_0^s A\dfrac{\partial }{\partial u}\psi_{t,u}(x)\,du \int_0^u X(x, v)\,dv\\
\stackrel{\eqref{eqinsr}}{=}\int_{\rr^d} \, dx\int_0^s A\psi_{t,u}(x)X(x, u)\,du
+ \int_0^s \, du\int_{\rr^d} A\dfrac{\partial }{\partial u}\psi_{t,u}(x)\,dx\int_0^u X(x, v)\,dv\\
\stackrel{\eqref{eqzuns}}{=} \int_{\rr^d} \, dx\int_0^s A\psi_{t,u}(x)X(x, u)\,du+
\int_0^s \,du\int_{\rr^d} \dfrac{\partial }{\partial u}\psi_{t,u}(x)X(x, u)\,dx\\
\stackrel{\eqref{eqinsr}}{=}\int_{\rr^d} \, dx\int_0^s \Bigl(A\psi_{t,u}(x)+\dfrac{\partial }{\partial
u}\psi_{t,u}(x)\Bigr)X(x, u)\,du=0 .
\end{multline*}

Passing to the limit as $t\downarrow s$ and applying Lemma~\ref{lmfnxr}, we arrive at
\[
\int_{\rr^d}X(x, s)\varphi (x)\, dx=0.\qedhere
\]
\end{proof}

\begin{exa}
Let stochastic measure $\mu$ be generated by a continuous square integrable martingale $Y$, $\mu(\aaa)=\int_0^T {\bf
1}_{\aaa}(t)\,dY(t)$, $\lambda$ be the Lebesgue measure on~$\rr^d$. Then $M_t(\aaa)=Y(t)\lambda(\aaa)$, $0\le t\le T$,
$\aaa\subset\rr^d$, is a worthy martingale measure with the dominating measure
\[
K(\aaa\times\bbb\times(0, t])=\bigl|\langle Y \rangle_t\bigr|\lambda(\aaa)\,\lambda(\bbb)
\]
(we use the terminology of~\cite{walsh}). In this case, \eqref{eqhbmusl} leads to
\begin{multline}\label{mltint}
X(x, t)=\int_{\rr^d} p(x, y, t)\xi(y)\,dy+\int_{[0, t]} \,d\mu(s)\int_{\rr^d} p(x, y, t-s)f(y,s)\,dy\\= \int_{\rr^d}
p(x, y, t)\xi(y)\,dy+\int_{[0, t]\times \rr^d} p(x, y, t-s)f(y,s)\,M(dy\, ds).
\end{multline}
The results of \cite[Chapter 2]{walsh} imply that the integral with respect to $M(dy\, ds)$ is well defined and is the
limit of integrals of elementary functions. For elementary function, equality of two stochastic integrals
in~\eqref{mltint} is obvious. Further, we can use the dominated convergence theorem for integral with respect to
$d\mu(s)$.

Similar solution of parabolic SPDE with respect to general martingale measure we have in Example~9 and
Remark~20~\cite{dalang}.
\end{exa}

\begin{exa} Assume that $\mu$ is generated by real-valued Wiener process $w$, $\jj$ denotes the set
of Schwartz rapidly decreasing test functions in $\rr^d$. Then equation
\[
\langle\ww(t),\psi\rangle=w(t)\int_{\rr^d} \psi(x)\,dx,\quad \psi\in\jj,
\]
defines the spatially homogeneous Wiener process with values in $\jj'$ (we used the terminology of~\cite{peszab97}).
For this case, our equality~\eqref{eqhbmusl} is a partial case of mild solution (2.6)~\cite{peszab97}.
\end{exa}

\begin{xrem}
By similar way, we can consider a more general equation
\begin{equation}\label{eqparg}
dX(x, t)=AX(x, t)\,dt+\sum_{1\le i\le j}f_i(x, t)d\mu_i(t),\quad X(0)=\xi,
\end{equation}
which includes the case
\[
dX(x, t)=AX(x, t)\,dt+f_1(x, t)dt+f_2(x, t)d\mu(t),\quad X(0)=\xi.
\]
The solution of~\eqref{eqparg} is
\[
X(x, t)=S(t)\xi(x)+\sum_{1\le i\le j}\int_{[0, t]} [S(t-s)f_i(x, s)]\,d\mu_i(s) .
\]
Under assumptions of Theorem~\ref{thhbmuinw}, the solution of~\eqref{eqparg} is unique.
\end{xrem}

\bigskip

\textsc{Department of Mathematical Analysis, Taras Shevchenko National University of Kyiv, Kyiv 01601, Ukraine}\\
\emph{E-mail adddress}: \verb"vradchenko@univ.kiev.ua"

\end{document}